\documentclass[a4paper,10pt]{article}
\usepackage[utf8]{inputenc}
\usepackage{amsmath}
\usepackage{amssymb}
\usepackage{amsthm}
\usepackage{dsfont}
\usepackage{hyperref}
\usepackage{graphicx}
\usepackage{enumerate}

\newcommand{\Pro}{\mathbb{P}}
\newcommand{\R}{\mathbb{R}}

\newtheorem{theorem}{Theorem}[section]
\newtheorem{definition}[theorem]{Definition}

\title{Bounds on Integrals with Respect to Multivariate Copulas}
\author{Michael Preischl\footnote{The author gratefully acknowledges support by the Austrian Science Fund (FWF) Project F5510 (part of the Special Research Program (SFB) "Quasi-Monte Carlo Methods: Theory and Applications")}}

\begin{document}

\maketitle

\begin{abstract}%evtl an jeweiliges Journal anpassen
Finding upper and lower bounds to integrals with respect to copulas is a quite prominent problem in applied probability. In their 2014 paper \cite{hofer2014optimal}, Hofer and Iac\'{o} showed how particular two dimensional copulas are related to optimal solutions of the two dimensional assignment problem. Using this, they managed to approximate integrals with respect to two dimensional copulas. In this paper, we will further illuminate this connection, extend it to $d$-dimensional copulas and therefore generalize the method from \cite{hofer2014optimal} to arbitrary dimensions. We also provide convergence statements. As an example, we consider three dimensional dependence measures.
\end{abstract}

\section{Introduction}
A multidimensional distribution function which has uniform margins is called a \textit{copula}. Over the last few decades, the study of copulas has enjoyed ever increasing popularity and a lot of progress has been made. For a big part this is due to their immediate practical relevance in modelling, especially in the context of financial mathematics. The intriguing idea is to split the description of a multidimensional random variable in its two components, the marginal laws and the copula which describes the dependence among the univariate random variables. The result that justifies this approach is \textit{Sklar's Theorem}, details can be found, e.g., in Nelsen's book \cite{nelsen2007introduction}.\\
In contrast to the often unsatisfying approach to quantify risk and dependence by single numbers like Spearman's $\rho$ or Kendall's $\tau$, modeling with copulas makes it possible to describe and incapsulate the whole dependence structure between random variables. On the other hand, an obvious downside of the copula method is that, unlike simple concordance measures, they can be rather hard to treat analytically, especially in dimensions higher than two. Currently, this problem is being circumvented by restricting to particular parametrized classes of copulas and by using structures consisting of multiple two dimensional copulas, so called \textit{vine copula constructions}, to approximate a distribution of higher dimension (see e.g. \cite{bedford2002vines}, \cite{de2010pair} and \cite{jaworski2010copula}).\\
However, these methods are often still computationally expensive or require a lot of simplifying assumptions that might or might not be justified. Hence, instead of modeling the actual dependence structure, it is naturally interesting to ask for ``worst case'' respectively ``best case'' behaviour.\\
Imagine that we are given a $d$ dimensional random vector $(X_1,\dots,X_d)$ and a function $f:\R^d\to\R$ that computes the quantity associated with $(X_1,\dots,X_d)$ which we wish to maximize or minimize (e.g. the VaR, some risk measure, an expected payoff etc.). We assume that we know the marginal distributions of $X_1,\dots,X_d$ whereas the dependence structure (i.e. the common distribution function) is completely unknown. This assumption is also called \textit{dependence uncertainty} and it is widely used in applications, mainly because compared to finding the dependence structure, information about the marginal laws can be obtained relatively easy.\\
By Sklar's Theorem it is always possible to reduce this maximization respectively minimization to a similar problem involving uniform marginals $X_1,\dots,X_d$, we again omit details and refer to \cite{mai2014financial} . Therefore, our aim is to find copulas $C_{\min}$ and $C_{\max}$ such that
\begin{equation}\label{1Cmin_def}
\int_{[0,1]^d}f(x_1,\dots,x_d)dC_{\min}\leq\int_{[0,1]^d}f(x_1,\dots,x_d)dC
\end{equation}
and
\begin{equation}\label{1Cmax_def}
\int_{[0,1]^d}f(x_1,\dots,x_d)dC\leq\int_{[0,1]^d}f(x_1,\dots,x_d)dC_{\max}
\end{equation} 
for all $d$ dimensional copulas $C$.\\
$ $\\
So far, most approaches were restricted to particular situations. Rüschendorf \cite{ruschendorf1980inequalities} for example considered so called $\Delta$-monotone functions, whereas Lux and Papapantoleon \cite{lux2016improved} focused on the case where partial information about the dependence structure is available. Another, quite different, take on this optimization problem that has become very famous in recent years is the so called \textit{Rearrangement Algorithm}, introduced by Puccetti and Rüschendorf \cite{puccetti2012computation} and further developed by Embrechts, Puccetti and Rüschendorf \cite{embrechts2013model}, \cite{puccetti2015computation}. This algorithm is impressingly efficient in approximating the desired bounds even in high dimensions and thus suffices for most real world applications. So far, however, it only works for supermodular functions and also lacks a rigorous proof of convergence. In fact there are known counterexamples, where the algorithm will not converge.\\
In two dimensions, Hofer and Iac\'{o} \cite{hofer2014optimal} combined the spirit of optimization theory with rigorous structural considerations and came up with an algorithm that converges to the minimal respectively maximal values of \eqref{1Cmin_def} and \eqref{1Cmax_def} for any continuous function $f$.  
% Erwähnen dass dieses Problem schon länger untersucht wird
The aim of this paper is to generalize the method from \cite{hofer2014optimal} to arbitrary dimensions.
\section{Mathematical Foundations}
We introduce the needed mathematical notions first. The most common and usually most practical definition of a copula is as multidimensional distribution function with uniform marginals.
\begin{definition}\textbf{(Copula)}
A function $C:[0,1]^d\to[0,1]$ is called a $d$-copula if there is a random vector $(U_1,\dots,U_d)$ such that for each $k=1,\dots,d$ the random variable $U_k$ has a uniform distribution on $[0,1]$ and
\[
C(u_1,\dots,u_d)=\Pro(U_1\leq u_1,\dots,U_d\leq u_d),\quad u_1,\dots,u_d\in[0,1].
\]
We write $\mathcal{C}^d$ for the set of all $d$-copulas.
\end{definition}
Note that every $d$-copula $C$ defines a measure $\mu_C$ on $([0,1]^d,\mathcal{B}([0,1]^d))$ by the following construction:\\
For every rectangle $R=[a_1,b_1]\times[a_2,b_2]\times\dots\times[a_d,b_d]\subset[0,1]^d$ set\\
$\mu_C(R)=V_C(R)$ with
\begin{equation}\label{constr_cvolume}
V_C(R)=\sum_{x\in\mathcal{V}(R)}sgn(x)C(x),
\end{equation}
where $\mathcal{V}(R)$ denotes the vertex set of $R$ and $sgn(x)$ is given as follows
\[
sgn(x)=\begin{cases}
        1, &\text{if $x_k=a_k$ for an even number of $k$'s}\\
        -1, &\text{if $x_k=a_k$ for an odd number of $k$'s.}
       \end{cases}
\]
Now extend $\mu_C$ to a measure on all Borel sets of $[0,1]^d$ by standard arguments of measure theory. $V_C(R)$ is called the $C$-volume of $R$.\\
$ $\\
As an example, for a $2$-copula $C$, the $C$-volume of a rectangle\\
$R=[a_1,b_1]\times[a_2,b_2]\subset[0,1]^2$ would be given as\\
$V_C(R)=C(b_1,b_2)-C(b_1,a_2)-C(a_1,b_2)+C(a_1,a_2)$.\\
\begin{definition}
Any probability measure $\mu$ on $([0,1]^d,\mathcal{B}([0,1]^d))$ that fulfills
\begin{equation}\label{d_fold_def}
\mu(\underbrace{[0,1]\times[0,1]\times\dots\times[0,1]}_{i-1\text{ times}}\times A\times \underbrace{[0,1]\times[0,1]\times\dots\times[0,1]}_{d-i\text{ times}})=\lambda(A),
\end{equation}
for all $i=1,\dots,d$ and any Borel set $A\subset [0,1]$ is called $d$-fold stochastic. Here, $\lambda(A)$ denotes the Lebesgue measure of A.
\end{definition}
It is easy to check that each measure $\mu_C$ that stems from a $d$-copula via \eqref{constr_cvolume} fulfills equation \eqref{d_fold_def} and is therefore $d$-fold stochastic. Furthermore, given any $d$-fold stochastic measure $\mu$, one can define a map\\
$C:[0,1]^d\to[0,1]$ by setting
\begin{equation}\label{copula_from_dfold}
C(x_1,x_2,\dots,x_d):= \mu([0,x_1]\times[0,x_2]\times\dots\times[0,x_d]).
\end{equation}
Again, it is easy to verify that this construction yields indeed a $d$-copula. Hence there is a one to one correspondence between $d$-copulas and $d$-fold stochastic measures.\\
In fact, Nelson \cite{nelsen2007introduction} even uses this property for an alternative, somewhat more general way to define copulas whereas, for us, it is going to be the cornerstone of the approach to relate copulas to assignment problems.\\
$ $\\
%%%%%% OPTIMAL TRANSPORT%%%%%%%%%%
We also want to introduce basic notions of optimal transport theory as it provides not only a nice framework for the copula optimization but will also be useful for convergence statements later on. The theory of optimal transport is concerned with the question how to allocate mass from some present distribution to a desired target distribution such that the total allocation costs are minimized. Historically, this question goes back to Monge, who formulated it in 1781. However his approach was very restrictive and so the field didn't see much progress until Kantorovich reformulated the question in 1942. Nowadays, the setting known as \textit{Monge-Kantorovich problem} has produced a huge amount of literature and ongoing research. The mathematical formulation is as follows
%Short Version
%We also want to introduce basic notions of optimal transport theory as it provides not only a nice framework for the copula optimization but will also be useful for convergence statements later on. The setting known as \textit{Monge-Kantorovich problem} has produced a huge amount of literature and ongoing research. The mathematical formulation is as follows
\begin{definition}
Given Polish spaces $X_1,\dots,X_d$ with probability measures\\
$\mu_1, \dots,\mu_d$ on their respective $\sigma$-fields, we write $\mathcal{M}(\mu_1,\dots,\mu_d)$ for the set of probability measures on the product space $X_1\times\cdots\times X_d$ which have marginal distributions $\mu_1,\dots,\mu_d$. Given a measurable cost function $c:X_1\times\cdots\times X_d\to\R$, the minimization problem
\begin{equation}\label{Monge_Kant_def}
\inf_{\mu\in\mathcal{M}(\mu_1,\dots,\mu_d)}\int c\text{ }d\mu
\end{equation}
is called the Monge-Kantorovich problem.
\end{definition}
The similarity to the problems \eqref{1Cmin_def} and \eqref{1Cmax_def} is obvious, we just need to prescribe the marginal laws $\mu_1,\dots,\mu_d$ as uniform distributions on $[0,1]$. Interestingly, the case $d=2$ is very well understood by now whereas higher dimensions seem to be harder to tackle. See for example R{\"u}schendorf and Uckelmann (\cite{ruschendorf1997optimal}, \cite{ruschendorf2002n}) for a treatment of so called \textit{couplings}, which are a probabilistic interpretation of the Monge-Kantorovich problem in higher dimensions.\\
$ $\\
In the two dimensional setting, the notion of \textit{$c$-cyclical monotonicity} proved to be very useful. %Wirklich c-cycl. doppelt definieren?
\begin{definition}
Given two polish spaces $X_1$ and $X_2$ and a measurable function $c:X_1\times X_2\to\R$, a set $\Gamma\subset X_1\times X_2$ is called c-cyclically monotone if for any pairs $(x_1,y_1),\dots,(x_n,y_n)\in \Gamma$ it holds that
\begin{equation}\label{ccyc_2dim_def}
\sum_{i=1}^nc(x_i,y_i)\leq\sum_{i=1}^nc(x_i,y_{i+1}),
\end{equation}
with $y_{n+1}=y_1$. A probability measure $\mu$ is called c-cyclically monotone if it is concentrated on a c-cyclically monotone set, i.e. if there is a c-cyclically monotone $\Gamma$ such that $\mu(\Gamma)=1$.
\end{definition}
Now a famous result in optimal transport theory states that, under mild assumptions on $c$, a probability measure $\mu$ is optimal for the two dimensional Monge-Kantorovich problem if and only if it is c-cyclically monotone and if we want to maximize instead, we can just reverse the inequality in \eqref{ccyc_2dim_def}. This optimality result follows from a dual formulation of the problem, which we do not want to state in detail but refer to the book of Villani \cite{villani2008optimal} which gives a very nice overview of optimal transport theory.\\
$ $\\
After finding a similar statement for dimensions higher than two has been an open problem for many years, Griessler and Beiglböck \cite{beiglbock2014optimality}, \cite{griessler2016c} recently generalized c-cyclical monotonicity to arbitrary dimensions:
\begin{definition}[\cite{beiglbock2014optimality} and \cite{griessler2016c}]
Let $X_1,\dots,X_d$ be Polish spaces and define $E:=X_1\times\dots\times X_d$. Let $c:E\to\mathbb{R}$ be Borel measurable. A set $\Gamma\subset E$ is called $c$-cyclically monotone if it fulfills one of the following conditions:
\begin{enumerate}[(i)]
 \item{For any $N$ and any points $(x_1^{(1)},\dots,x_d^{(1)}),\dots,(x_1^{(N)},\dots,x_d^{(N)})\in\Gamma$ and permutations $\sigma_2,\dots,\sigma_d:\{1,\dots,N\}\to\{1,\dots,N\}$, one has
 \[
  \sum_{i=1}^Nc(x_1^{(i)},\dots,x_d^{(i)})\leq\sum_{i=1}^Nc(x_1^{(i)},x_2^{(\sigma_2(i))},\dots,x_d^{(\sigma_d(i))}).
 \]}
 \item{Any finite measure $\alpha$ concentrated on finitely many points in $\Gamma$ is a (with respect to $c$) cost-minimizing transport plan between its marginals; i.e. if $\alpha'$ has the same marginals as $\alpha$, then
 \[
  \int c d\alpha \leq \int c d\alpha'.
 \]}
\end{enumerate}
\end{definition}
They were also able to show that for any measurable cost function $c$, a measure $\mu$ which is optimal for the multidimensional Monge-Kantorovich problem is always concentrated on some c-cyclically monotone set.\\
Griessler \cite{griessler2016c} recently showed the converse statement under more assumptions on $c$: If the cost function $c$ is continuous and bounded by a sum of integrable functions, any measure which is  concentrated on a c-cyclically monotone set is an optimal solution to \eqref{Monge_Kant_def}.\\
$ $\\
%%%%%ASSIGNMENT PROBLEMS%%%%%%%
Next we give a short overview of assignment problems. Generally speaking, an assignment problem consists of the task of matching $n$ objects of type $1$ to $n$ objects of type $2$. Here, matching object $i$ of type $1$ with object $j$ of type $2$ creates a cost (resp. profit) of $a_{ij}\in\R$. Now the problem is to minimize (resp. maximize) the sum over all $a_{ij}$ under the following constraints:
\begin{itemize}
\item{Each object of type $1$ is assigned with exactly one object of type $2$.}
\item{Each object of type $2$ is assigned with exactly one object of type $1$.}
\end{itemize}
This kind of problem is often described by assigning ``jobs'' to ``workers'' such that the total cost is minimized. The mathematical formulation is the following
\begin{equation}\label{2AP_objective_function}
\min_{x\in\R^{n\times n}}\sum_{i=1}^n\sum_{j=1}^na_{ij}x_{ij}
\end{equation}
subject to
\begin{align}
\sum_{j=1}^nx_{ij}=1\qquad\forall i\in\{1,\dots,n\}\label{2AP_constraint1},\\
\sum_{i=1}^nx_{ij}=1\qquad\forall j\in\{1,\dots,n\}\label{2AP_constraint2},\\
x_{ij}\in\{0,1\}.\label{2AP_integer_constraint}
\end{align}
The expression that is to be minimized (resp. maximized) in \eqref{2AP_objective_function} is also called the \textit{objective function} and the set of all $x\in\R^{n\times n}$ that fulfill all the constraints are called the \textit{feasible region}.\\
It is not hard to see that this can equivalently be written in the form
\[
\min_{\pi\in S_n}\sum_{i=1}^na_{i\pi(i)},
\]
where $S_n$ denotes the set of all permutations of $\{1,\dots,n\}$. Although the feasible region of this problem is actually $n^2$ dimensional, with $n$ being the number of objects, we will refer to this version of the assignment problem as the ``$2$ dimensional assignment problem (2-AP)'' since there are $2$ different kinds of objects. The assignment problem is quite well studied and, at least for the $2$ dimensional case, there are efficient algorithms like, for example, the Hungarian Method, which, in current implementation, has a polynomial time complexity of $\mathcal{O}(n^3)$. For more details we refer to \cite{burkard2009assignment}.\\
$ $\\
There is a class of copulas, which is closely related to assignment problems, the so-called \textit{Shuffles of M} which we introduce briefly.
\begin{definition}\textbf{(Shuffles of $M$ \cite{durante2009shuffles})}
Let $n\geq1$ be a positive integer and $s=(s_0,\dots,s_n)$ be a partition of the unit interval with $0=s_0<s_1\dots<s_n=1$ Furthermore, let $\pi$ be a permutation of $\{1,\dots,n\}$ and define the partition $t=(t_0,\dots,t_n)$ with $0=t_0<t_1\dots<t_n=1$ such that each $[s_{i-1},s_i)\times[t_{\pi(i)-1},t_{\pi(i)})$ is a square. Finally, let $\omega:\{1,\dots,n\}\to\{-1,1\}$. A copula $C$ is called Shuffle of $M$ with parameters $\{n,s,\pi,\omega\}$ if it is defined in the following way: for all $i\in\{1,\dots,n\}$, if $\omega(i)=1$ then $C$ distributes a mass of $s_i-s_{i-1}$ uniformly spread along the diagonal of $[s_{i-1},s_i)\times[t_{\pi(i)-1},t_{\pi(i)})$ and if $\omega(i)=-1$ then $C$ distributes a mass of $s_i-s_{i-1}$ uniformly spread along the antidiagonal of $[s_{i-1},s_i)\times[t_{\pi(i)-1},t_{\pi(i)})$.
\end{definition}
In the case of an equidistant partition $s$, i.e. $s_i-s_{i-1}=\frac1n$ for all $i\in\{1,\dots,n\}$ we obviously have $s=t$ for all permutations $\pi$. We write $\mathcal{C}^d_M$ for the set of all $d$-dimensional Shuffles of $M$. It is important to notice (see e.g. \cite{durante2009shuffles}) that $\mathcal{C}^d_M\subset \mathcal{C}^d$, i.e. the above construction yields indeed again a copula. Another neat property of shuffles of $M$ is that they lie dense in the set of all copulas: $\overline{\mathcal{C}^d_M}=\mathcal{C}^d$ with respect to weak convergence. For more details and a survey of different metrics see \cite{durante2012approximation}.\\
$ $\\
We are now ready to state the main result from \cite{hofer2014optimal} that connects Shuffles of $M$ and assignment problems to integrals with respect to copulas.
\begin{theorem}[Theorem 2.1 + Theorem 2.2 from \cite{hofer2014optimal}]
Let $f$ be a continuous function on $[0,1]^2$ and let the partition $I^n$  for any $n$ be given as
\[
I_{ij}^n:=\left[\frac{i-1}{n},\frac{i}{n}\right)\times\left[\frac{j-1}{n},\frac{j}{n}\right)\quad\text{for }i,j=1,\dots,n.
\]
Then define
\[
f_n(x_1,x_2)=a_{ij}:=\max_{(x_1,x_2)\in I_{ij}^n}f(x,y)\qquad\forall (x_1,x_2)\in I_{ij}^n.
\]
Now a copula $C_n^{max}$ which fulfills
\begin{equation}\label{cop_assign_thm_eq1}
\int_{[0,1]^2}f_n(x_1,x_2)dC_n^{max}=\max_{C\in\mathcal{C}^2}\int_{[0,1]^2}f_n(x_1,x_2)dC(x_1,x_2)
\end{equation}
is given as a shuffle of $M$ with parameters $(n,I^n,\pi^*,1)$ where $\pi^*$ is the permutation which solves the assignment problem
\[
\max_{\pi\in S_n}\sum_{i=1}^na_{i\pi(i)}.
\]
Moreover, the maximal value of \eqref{cop_assign_thm_eq1} is given as
\[
\max_{C\in\mathcal{C}^2}\int_{[0,1]^2}f(x_1,x_2)dC(x_1,x_2)=\frac1n\sum_{i=1}^na_{i\pi^*(i)}
\]
and it holds
\[
\lim_{n\to\infty}\max_{C_n\in\mathcal{C}^2}\int_{[0,1]^2}f_n(x_1,x_2)dC_n(x_1,x_2)=\max_{C\in\mathcal{C}^2}\int_{[0,1]^2}f(x_1,x_2)dC(x_1,x_2).
\]
\end{theorem}
$ $\\
Furthermore, Iac\'{o} et al. \cite{iaco2015distribution} showed that the sequence of maximizers $C_n^{max}$ converges, at least along some subsequence, to a maximizer $C^*$ of the problem
\[
\max_{C\in\mathcal{C}^2}\int_{[0,1]^2}f(x_1,x_2)dC(x_1,x_2).
\]

\section{Main Results}
Our main result is a version of Theorems $2.1$ and $2.2$ from \cite{hofer2014optimal} for arbitrary dimensions along with a similar convergence result as Theorem 4.2 from \cite{iaco2015distribution}. To this end we start by introducing the concept of a multidimensional assignment problem.\\
$ $\\
%Similarly as in \cite{hofer2014optimal} we want to start by looking at cost functions $f:[0,1]^3\to\R$ that are constant on the cubes $I_{i,j,k}^n$ with
%\[
%I_{i,j,k}^n:=\left[\frac{i-1}{n},\frac{i}{n}\right)\times\left[\frac{j-1}{n},\frac{j}{n}\right)\times\left[\frac{k-1}{n},\frac{k}{n}\right).
%\]
%So $f(x_1,x_2,x_3)=c_{i,j,k}$ for all $(x_1,x_2,x_3)\in I_{i,j,k}$ and we wish to maximize the expression
%\begin{equation}\label{3dim_cop_max}
%\int_{[0,1)^3} f(x_1,x_2,x_3) dC(x_1,x_2,x_3)
%\end{equation}
%over $\mathcal{C}_3$, the set of all $3$-copulas.\\
%$ $\\
Define the index sets $\mathcal{I}:=\{1,\dots,n\}^d$ and $\mathcal{I}^k_m:=\{(i_1,\dots,i_d)\in\mathcal{I}:i_k=m\}$. The (axial) $d$ dimensional assignment problem ($d$-AP) on $n$ items with objective function $(a_{i})$ is given as follows:
\begin{equation}\label{dAP_objective_function}
\min_{x\in\R^{n^d}}\sum_{i\in\mathcal{I}}a_{i}x_{i} 
\end{equation}
subject to
\begin{align}
&\sum_{i\in\mathcal{I}^k_m}x_i=1,\qquad \forall m\in\{1,\dots,n\},\quad\forall k\in\{1,\dots,d\}, \label{3AP_constraint1}\\
&x_{i}\in\{0,1\}.\label{3AP_integer_constraint}
\end{align}
Again, ``$d$ dimensional'' is meant with respect to the number of different object types. The feasible region in this case would actually be $n^d$ dimensional.\\
$ $\\
The key to generalizing the method from \cite{iaco2015distribution} to higher dimensions now lies in recognizing the structural analogy between copulas and assignment problems.
\begin{theorem}\label{structure_analogy}
Let $n$ be a positive integer and $f:[0,1]^d\to\R$ be constant on the cubes $I_{i}^n$ with
\[
I_{i}^n:=\left[\frac{i_1-1}{n},\frac{i_1}{n}\right)\times\dots\times\left[\frac{i_d-1}{n},\frac{i_d}{n}\right)
\]
for $i=(i_1,\dots,i_d)\in\mathcal{I}$. So $f(x)=a_{i}$ for all $(x)\in I_{i}$.\\
Then a copula $C^{min}$ which fulfills
\begin{equation}\label{3dim_cop_max}
\int_{[0,1]^d}f(x)dC^{min}=\min_{C\in\mathcal{C}^d}\int_{[0,1]^d} f(x) dC(x)
\end{equation}
is given by the following $d$-fold stochastic measure on $[0,1]^d$:\\
If $(x^*_{i})$ with $i\in\mathcal{I}$ is an optimal solution to the relaxed $d$-AP with respect to the objective function $(a_{i})$, then distribute the probability mass $\frac{1}{n}x^*_{i}$ uniformly in the cube $I_{i}$.\\
\end{theorem}
Here ``relaxed'' means that we are considering the continuous relaxation of the axial $d-AP$, i.e. we replace the integer constraint \eqref{3AP_integer_constraint} by just demanding $x_{i}\geq0$ for all $i$. 
$ $\\
This result holds for any dimension $d$. However, in contrast to the $2$ dimensional case, the maximizing copula that comes from the assignment problem will, for $d>2$, in general not be given as a Shuffle of $M$ anymore. This is because any feasible (not necessarily optimal) solution of the \textit{relaxed} $d-AP$ yields a $d$ fold stochastic measure and hence a copula and only the points with integer coordinates of the feasible region yield a Shuffle of $M$.\\
In fact, the reason why there will always be a Shuffle of $M$ as a maximizer of the integral in the $2$ dimensional case lies in a certain peculiarity of the relaxed $2$ dimensional assignment problem: in $2$ dimensions, all the vertices of the feasible region have only $0$ or $1$ entries. That means constraint \eqref{2AP_integer_constraint} is redundant, or in other words the $2$ dimensional assignment problem is identical to its continuous relaxation. This is not true anymore in any dimension greater than $2$, we again refer to \cite{burkard2009assignment} for further details about assignment problems.\\
$ $\\
That means, for dimensions greater than $2$, the maximizer we find via this procedure will not have the nicely parametrized form, that made Shuffles of $M$ very appealing. On the other hand, working with the relaxed assignment problem instead of the integer problem brings great advantages concerning computability. While the classical integer assignment problems are $\mathcal{NP}$ hard, their continuous relaxations lie in $\mathcal{P}$, here the computational complexity is with respect to the number of objects that should be assigned, or in the context of copulas, the coarseness of the partition of $[0,1]^d$.
\begin{proof}[Proof of Theorem \ref{structure_analogy}]
By definition, the value of \eqref{3dim_cop_max} is given as
\begin{align*}
 \int_{[0,1]^d} f(x) dC(x)=\sum_{i\in\mathcal{I}}a_{i}C(I^n_{i}).
\end{align*}
%\sum_{w\in\mathcal{H}(f)}w\text{ }C(\{v:f(v)=w\})
%where $\mathcal{H}(f):=\{w:\exists v$ with $f(v)=w\}$.
So we notice that the value of \eqref{3dim_cop_max} does not depend on how the copula $C$ distributes mass inside of each cube $I^n_{i}$, but only on how much mass is placed on each $I^n_{i}$. Hence, we can write $C^{min}(I^n_{i})=:x_{i}$, with $i\in\mathcal{I}$ and are left with the following optimization problem: % evtl Bezeichnung Ändern, c_ijk x_ijk nicht optimal
\[
 \min\sum_{i\in\mathcal{I}}a_{i}x_{i}.
\]
However, we still have to encode constraints that ensure that the mass distribution $x_{i}$ actually yields a copula. For this, we recall that there is a one to one correspondence between $d$-copulas and $d$-fold stochastic measures so we only have to ensure that the measure $C^{min}$ is $d$-fold stochastic. Since we already noted that the value of \eqref{3dim_cop_max} is independent of the distribution inside the cubes $I_{i}$, we can assume that $C^{min}$ distributes the mass inside of each cube $I_{i}$ uniformly. The $d$-fold stochastic measures which distribute the mass $x_{i}$ uniformly inside of the cube $I_{i}$ for each $i\in\mathcal{I}$ are given by the equations
\begin{equation}\label{3fold_cond1}
\sum_{i\in\mathcal{I}^k_m}x_i=\frac1n,\qquad \forall m\in\{1,\dots,n\},\quad\forall k\in\{1,\dots,d\}.
\end{equation}
This can be seen as follows: It is elementary that each $d$-fold stochastic measure satisfies the conditions \eqref{3fold_cond1}. So let $C$ fulfill \eqref{3fold_cond1} and let $0\leq a<b\leq1$. Now look for $1\leq i_-\leq i^+\leq n$ such that $a\in[\frac{i_--1}{n},\frac{i_-}{n}]$ and $b\in[\frac{i^+-1}{n},\frac{i^+}{n}]$. Without loss of generality let us consider the first coordinate, it holds 
\begin{align*}
 C([a,b]\times[0,1]\times\dots\times[0,1])=&C\left(\left[a,\frac{i_-}{n}\right]\times[0,1]\times\dots\times[0,1]\right)\\
&+\sum_{i_1=i_-+1}^{i^+-1}C\left(\left[\frac{i_1-1}{n},\frac{i_1}{n}\right]\times[0,1]\times\dots\times[0,1]\right)\\
&+ C\left(\left[\frac{i^+-1}{n},b\right]\times[0,1]\times\dots\times[0,1]\right)\\
&=\frac{i_-}{n}-a+\left(\sum_{i_1=i_-+1}^{i^+-1}\sum_{j\in\mathcal{I}^1_{i_1}}x_j\right) + b-\frac{i^+-1}{n}\\
%&=\frac{\frac{i_-}{n}-a}{\frac1n}\frac1n+\sum_{i_1=i_-+1}^{i^+-1}\sum_{j\in\mathcal{I}^1_{i_1}}x_j + \frac{b-\frac{i^++1}{n}}{\frac1n}\frac1n\\
&=b-a=\lambda([a,b]).
\end{align*}
Again by standard arguments of measure theory, we can extend this result from intervals to arbitrary measurable sets $A$.\\
Hence the measure $C^{min}$ is $d$-fold stochastic if and only if the constraints \eqref{3fold_cond1} are fulfilled. But those are clearly just the constraints \eqref{3AP_constraint1} from the $d$AP with the right hand side $\frac1n$ instead of $1$. Since scaling the right hand side of a linear optimization problem just results in a likewise scaling of the optimal solution, the optimal probability mass distribution $(x_{i})$ is given as $(x_{i})=\frac1n(x^*_{i})$ with $(x^*_{i})$ being the optimal solution to the general $d$-AP (i.e. right hand side $1$) with objective function $(a_{i})$. The representation of $C^{min}$ as a distribution function can be calculated via \eqref{constr_cvolume}. Exemplary for the case $d=3$ this looks as follows:
\begin{align*}
&\text{for } (x,y,z)\in [0,1]^3\text{ identify }i,j,k\text{ such that } (x,y,z)\in I^n_{(i,j,k)}\text{ and set }\\
&C^{min}(x,y,z):=C^{min}([0,x)\times[0,y)\times[0,z))=\sum_{i'<i,j'<j,k'<k}x_{(i',j',k')}\\
&+\sum_{i'=i,j'<j,k'<k}n\left(x-\frac{i-1}{n}\right)x_{(i,'j',k')} + \sum_{i'<i,j'=j,k'<k}n\left(y-\frac{j-1}{n}\right)x_{(i',j',k')}\\
&+\sum_{i'<i,j'<j,k'=k}n\left(z-\frac{k-1}{n}\right)x_{(i',j',k')}\\
&+\sum_{i'=i,j'=j,k'<k}n\left(x-\frac{i-1}{n}\right)n\left(y-\frac{j-1}{n}\right)x_{(i',j',k')}\\
&+\sum_{i'=i,j'<j,k'=k}n\left(x-\frac{i-1}{n}\right)n\left(z-\frac{k-1}{n}\right)x_{(i',j',k')}\\
&+\sum_{i'<i,j'=j,k'=k}n\left(y-\frac{j-1}{n}\right)n\left(z-\frac{k-1}{n}\right)x_{(i',j',k')}\\
&+n\left(x-\frac{i-1}{n}\right)n\left(y-\frac{j-1}{n}\right)n\left(z-\frac{k-1}{n}\right)x_{(i,j,k)}.
%$ $\\
\end{align*}
%Next, we note that for a measure $\mu$ that distributes the mass $x_{ijk}$ uniformly in $I_{ijk}$ the constraints \eqref{3AP_constraint1} - \eqref{3AP_constraint2} are equivalent to demanding that $\mu$ is $3$-fold stochastic.\\ %Genauer erklären?
%$ $\\

\end{proof}
From what we used in the proof it is quite obvious that Theorem \ref{structure_analogy} is equally valid for a maximization instead of a minimization.\\
$ $\\
The second result from \cite{hofer2014optimal}, namely the possibility to approximate the integral over continuous functions by a sequence of integrals over functions that are piecewise constant carries over practically verbatim.
\begin{theorem}\label{main_result}
Let $f$ be continuous on $[0,1]^d$ and bounded by a sum of integrable functions and let the sets $I_{i}^n$ for $i\in\mathcal{I}$ be given as before. Then, set
\begin{align*}
f^{\text{max}}_n(x):=\max_{y\in I_{i}^n}f(y)\qquad\forall x\in I_{i}^n\\
f^{\text{min}}_n(x):=\min_{y\in I_{i}^n}f(y)\qquad\forall x\in I_{i}^n.
\end{align*}
Now denote by $C^{\text{max}}_n$ and $C^{\text{min}}_n$ copulas which minimize
\begin{equation}\label{3d_upperbd_transport}
\min_{C\in\mathcal{C}^d}\int_{[0,1]^d}f^{\text{max}}_n(x)dC(x)
\end{equation}
and
\begin{equation*}
\min_{C\in\mathcal{C}^d}\int_{[0,1]^d}f^{\text{min}}_n(x)dC(x).
\end{equation*}
Then
\begin{align*}
&\lim_{n\to\infty}\int_{[0,1]^d}f^{\text{min}}_n(x)dC^{\text{min}}_n(x)= \lim_{n\to\infty}\int_{[0,1]^d}f^{\text{max}}_n(x) dC^{\text{max}}_n(x)\\
&=\inf_{C\in\mathcal{C}^d}\int_{[0,1]^d}f(x)dC(x).
\end{align*}
Furthermore, the sequences of maximizers $C^{\text{max}}_n$ and $C^{\text{min}}_n$ converge, at least along some subsequence, to a maximizer $C^*$ of the problem
\begin{equation*}
\min_{C\in\mathcal{C}_d}\int_{[0,1]^d}f_n(x)dC(x).
\end{equation*}\\
\end{theorem}
\begin{proof}
Just as in \cite{hofer2014optimal}, we show directly that
\[
 \lim_{n\to\infty}\int_{[0,1]^d}f^{\text{min}}_n(x)dC^{\text{min}}_n(x) = \lim_{n\to\infty}\int_{[0,1]^d}f^{\text{max}}_n(x) dC^{\text{max}}_n(x)
\]
and thus get the claim.\\
$ $\\
Let $\varepsilon>0$. Since $f$ is continuous, we can choose $n$ large enough such that
\[
 |f^{\text{min}}_n(x)-f^{\text{max}}_n(x)|<\varepsilon\qquad\forall x\in [0,1]^3.
\]
Furthermore, $f^{\text{min}}_n$ is piecewise constant, so we have
\begin{align*}
 \int_{[0,1]^d} f^{\text{min}}_n(x) dC^{\text{min}}_n(x)&=\sum_{i\in\mathcal{I}}a_{i}C^{\text{min}}_n(I_{i}),
\end{align*}
with $a_{i}:=\min_{x\in I_{i}^n}f(x)$. Hence
\begin{align*}
 &\int_{[0,1]^d} f^{\text{max}}_n(x) dC^{\text{max}}_n(x) < \int_{[0,1]^d}\left(f^{\text{min}}_n(x)+\varepsilon\right)\text{ } dC^{\text{min}}_n(x)=\sum_{i\in\mathcal{I}}C^{\text{min}}_n(I_{i})(a_{i}+\varepsilon).
\end{align*}
So we have
\begin{align*}
 &\left|\int_{[0,1]^d} f^{\text{max}}_n(x) dC^{\text{max}}_n(x)-\int_{[0,1]^d} f^{\text{min}}_n(x)dC^{\text{min}}_n(x)\right|\\
 &<\left|\sum_{i\in\mathcal{I}}\varepsilon\text{ }C^{\text{min}}_n(I_{i})\right|=\varepsilon.
\end{align*}
The proof that the sequence of optimizers converges to an optimizer for the continuous function follows largely the arguments of the proof of Theorem $4.2$ in \cite{iaco2015distribution}. To begin with, we can use Theorem $5.21$ from \cite{kallenberg2006foundations} to deduce that $C^{\text{max}}_n$ converges weakly, at least along some subsequence, to a copula $C^*$.\\ % Prüfen ob Menge der Copulas wirklich abgeschlossen ist.
$ $\\
According to Theorem $2.4$ from \cite{beiglbock2014optimality}, any measure which is an optimal solution to a transportation problem is necessarily concentrated on a c-cyclical monotone set. So since $C^{\text{max}}_n$ is optimal for the transportation problem \eqref{3d_upperbd_transport} it must be concentrated on a $c$-cyclical monotone set. Hence, for any $N\in\mathbb{N}$, the $N$-fold product measure $C^{\text{max},\otimes N}_n$ is concentrated on the set $\mathcal{S}_n(N)$ of points $(x_1^{(1)},\dots,x_d^{(1)}),\dots,(x_1^{(N)},\dots,x_d^{(N)})$ for which
\[
 \sum_{i=1}^Nf^{\text{max}}_n(x_1^{(i)},\dots,x_d^{(i)})\leq \sum_{i=1}^Nf^{\text{max}}_n(x_1^{(i)},x_2^{\sigma_2(i)},\dots,x_d^{\sigma_d(i)}).
\]
Now fix $\varepsilon>0$. Since $f$ is continuous, we can choose $n$ large enough such that $C^{\text{max},\otimes N}_n$ is concentrated on the set $\mathcal{S}_\varepsilon(N)$ of points with
\[
 \sum_{i=1}^Nf(x_1^{(i)},\dots,x_d^{(i)})\leq\sum_{i=1}^Nf(x_1^{(i)},x_2^{\sigma_2(i)},\dots,x_d^{\sigma_d(i)})+\varepsilon.
\]
The set $\mathcal{S}_\varepsilon(N)$ is closed, since $f$ is continuous. Therefore, also the limiting measure $C^{*,\otimes N}$ is concentrated on $\mathcal{S}_\varepsilon(N)$ for all $\varepsilon>0$. Now let $\varepsilon\to 0$ and we have that $C^{*,\otimes N}$ is concentrated on a set of points with % evtl Beweis noch genauer mit Stefan durchsprechen
\[
 \sum_{i=1}^Nf(x_1^{(i)},\dots,x_d^{(i)})\leq\sum_{i=1}^Nf(x_1^{(i)},x_2^{\sigma_2(i)},\dots,x_d^{\sigma_d(i)}),
\]
which means that $C^*$ is concentrated on a $c$-cyclically monotone set. Since $f$ is continuous and bounded by a sum of integrable functions, we can apply Theorem $1.2$ from \cite{griessler2016c} to deduce that $C^*$ is optimal. The proof for the sequence $C_n^{min}$ is identical.
\end{proof}
In addition, we also want to mention that the results about c-cyclical monotonicity from \cite{griessler2016c} enable us to generalize Theorem $3$ from Schachermayer and Teichmann \cite{schachermayer2009characterization} to arbitrary dimensions.
\begin{theorem}
Let $X_1,\dots,X_d$ be polish spaces and write $E:=X _1\times\dots\times X_d$. Let $c:E\to\R$ be continuous and bounded by a sum of integrable functions. Given probability measures $\mu^i$ on $X_i$ for $i=1,\dots,d$, assume that there are sequences $(\mu^i_n)_{n\geq1}$ of Borel probability measures converging weakly to $\mu^i$ on $X_i$ for $i=1,\dots,d$. Furthermore, let $\pi_n$ be an optimal solution to the Monge Kantorovich problem %Hier wirklich allg. Monge Kant reinbringen?
\[
\inf_{\pi\in\mathcal{M}(\mu^1_n,\dots,\mu^d_n)}\int_Ec\text{ } d\pi
\]
for $n\geq1$. Then there is a subsequence $(\pi_{n_k})_{k\geq1}$ which converges weakly to a measure $\pi$ and $\pi$ is optimal for the Monge-Kantorovich problem
\begin{equation}\label{mongo_woisopt}
\inf_{\pi\in\mathcal{M}(\mu^1,\dots,\mu^d)}\int_Ec\text{ } d\pi.
\end{equation}
Any other converging subsequence of $(\pi_n)_{n\geq1}$ also converges to an optimizer of the Monge-Kantorovich problem \eqref{mongo_woisopt}.
\end{theorem}
\begin{proof}
The proof is the same as for Theorem $3$ in \cite{schachermayer2009characterization}, replacing the $2$ dimensional notion of c-cyclical monotonicity with its multidimensional counterpart.
\end{proof}
Just as Iac\'{o} at al. did in \cite{iaco2015distribution}, we remark that this result can be used when the distributions $\mu^1,\dots,\mu^d$ are approximated by random samples and the corresponding empirical distributions.\\
$ $\\
The presented method generalizes the approach from \cite{hofer2014optimal} and furthermore facilitates the computation, since the relaxed assignment problem is much easier to solve than the integer one. As the complexity increases polynomially in $n$, relatively good approximations are already possible with very little requirements concerning hardware and implementation. However the complexity is exponential in the dimension of the copula, therefore, although theoretically still valid, this method will not be viable in practice for dimensions higher than say $d=10$. So for most practical applications, the rearrangement algorithm by Pucetti and Rüschendorf will, whenever applicable, still be the method of choice. The merit of this new approach lies in the generality of the statement. It is not limited to supermodular functions but requires merely continuity and a notion of boundedness, both of these assumptions might possibly be relaxed as research in the field of optimal transport progresses. Following the spirit of Lux and Papapantoleon, one might also consider including partial information about the distribution by simply adding suitable constraints to the linear program.
\section{Application to Dependence Measures}
A very neat field of application for this technique is the approximation of upper and lower bounds for dependence measures. In the bivariate case, there are well-known and widely used risk measures like for example \textit{Spearman's $\rho$}, \textit{Kendall's $\tau$}, \textit{Blomqvists $\beta$} and \textit{Gini's $\gamma$}. See e.g. \cite{mai2014financial} for the precise definitions.\\
Now, somewhat surprisingly, there is much less literature on multivariate extensions of these canonical dependence measures. Probably the main reason for this is that many concepts just do not have a multivariate analogy or do not allow for the same interpretation. As a result, there are different multivariate versions of most of the classical bivariate measures and the academic community has not yet settled on a uniform convention for most of these. In \cite{schmid2010copula} the authors gather the most prominent multivariate versions of the most important bivariate dependence measures and also provide a discussion of the advantages and disadvantages for each approach.\\
$ $\\
To apply our optimization method, we chose to look at a multi dimensional version of Spearman's $\rho$. Define
\begin{equation}\label{rho2}
\rho(C):=\frac{d+1}{2^d-(d+1)}\left(2^d\int_{[0,1]^d}\Pi(u)dC(u)-1\right).
\end{equation}
Here, $\Pi$ denotes the independence copula, i.e. $\Pi(x_1,x_2,x_3)=x_1x_2x_3$. Consider
\[
l_d:=\frac{2^d-(d+1)!}{d!(2^d-(d+1))},\qquad d\geq2.
\]
According to \cite{nelsen1996nonparametric}, it holds that $l_d$ is a lower bound on $\rho(C)$ but it is unknown whether this bound is best possible. We want to use Theorem \ref{main_result} to approximate
\[
\min_{C\in\mathcal{C}^3}\int_{[0,1]^3}\Pi(u)dC(u).
\]
So the first step is the discretization. We are interested in a lower bound on $\rho(C)$, so we look at
\[
\Pi_n^{min}(x):=\min_{y\in I_{(i_1,i_2,i_3)}^n}\Pi(y)=\frac{(i_1-1)(i_2-1)(i_3-1)}{n^3}
\]
for all $x\in I_{(i_1,i_2,i_3)}^n=\left[\frac{i_1-1}{n},\frac{i_1}{n}\right)\times\left[\frac{i_2-1}{n},\frac{i_2}{n}\right)\times\left[\frac{i_3-1}{n},\frac{i_3}{n}\right)$, since obviously
\[
\min_{C\in\mathcal{C}^3}\int_{[0,1]^3}\Pi_n^{min}(u)dC(u)\leq\min_{C\in\mathcal{C}^3}\int_{[0,1]^3}\Pi(u)dC(u)
\]
and
\[
\lim_{n\to\infty}\min_{C\in\mathcal{C}^3}\int_{[0,1]^3}\Pi_n^{min}(u)dC(u)=\min_{C\in\mathcal{C}^3}\int_{[0,1]^3}\Pi(u)dC(u).
\]
Now simple computations in $R$\footnote{This result was obtained using the ``lpSolve'' package for the open source program $R$. This package is built on the free Mixed Integer Linear Program solver lp\_solve, which utilizes the revised simplex method and the Branch-and-bound method. No presolve routines or any other kind of advanced techniques were used.} yield for a grid of $n=40$ sections in each dimension
\[
\min_{C\in\mathcal{C}^3}\int_{[0,1]^3}\Pi_{40}^{min}(u)dC(u)\approx -0.625,
\]
which, in other words, means the aforementioned lower bound $l_3=-\frac23$ cannot be attained, it is not best possible.

\bibliography{integral_bounds_bib}

\begin{thebibliography}{10}

\bibitem{bedford2002vines}
Tim Bedford and Roger~M Cooke.
\newblock Vines: A new graphical model for dependent random variables.
\newblock {\em Annals of Statistics}, pages 1031--1068, 2002.

\bibitem{beiglbock2014optimality}
Mathias Beiglb{\"o}ck and Claus Griessler.
\newblock An optimality principle with applications in optimal transport.
\newblock {\em arXiv preprint arXiv:1404.7054}, 2014.

\bibitem{burkard2009assignment}
Rainer~E Burkard, Mauro Dell'Amico, and Silvano Martello.
\newblock {\em Assignment Problems, Revised Reprint}.
\newblock Siam, 2009.

\bibitem{de2010pair}
Beatriz~Vaz de~Melo~Mendes, Mari{\^a}ngela~Mendes Semeraro, and Ricardo
  P~C{\^a}mara Leal.
\newblock Pair-copulas modeling in finance.
\newblock {\em Financial Markets and Portfolio Management}, 24(2):193--213,
  2010.

\bibitem{durante2012approximation}
Fabrizio Durante and Juan~Fern{\'a}ndez S{\'a}nchez.
\newblock On the approximation of copulas via shuffles of min.
\newblock {\em Statistics \& Probability Letters}, 82(10):1761--1767, 2012.

\bibitem{durante2009shuffles}
Fabrizio Durante, Peter Sarkoci, and Carlo Sempi.
\newblock Shuffles of copulas.
\newblock {\em Journal of Mathematical Analysis and Applications},
  352(2):914--921, 2009.

\bibitem{embrechts2013model}
Paul Embrechts, Giovanni Puccetti, and Ludger R{\"u}schendorf.
\newblock Model uncertainty and var aggregation.
\newblock {\em Journal of Banking \& Finance}, 37(8):2750--2764, 2013.

\bibitem{griessler2016c}
Claus Griessler.
\newblock $ c $-cyclical monotonicity as a sufficient criterion for optimality
  in the multi-marginal monge-kantorovich problem.
\newblock {\em arXiv preprint arXiv:1601.05608}, 2016.

\bibitem{hofer2014optimal}
Markus Hofer and Maria~Rita Iac{\`o}.
\newblock Optimal bounds for integrals with respect to copulas and
  applications.
\newblock {\em Journal of Optimization Theory and Applications},
  161(3):999--1011, 2014.

\bibitem{iaco2015distribution}
Maria~Rita Iac{\`o}, Stefan Thonhauser, and Robert~F Tichy.
\newblock Distribution functions, extremal limits and optimal transport.
\newblock {\em arXiv preprint arXiv:1502.06839}, 2015.

\bibitem{jaworski2010copula}
Piotr Jaworski, Fabrizio Durante, Wolfgang~Karl Hardle, and Tomasz Rychlik.
\newblock {\em Copula theory and its applications}.
\newblock Springer, 2010.

\bibitem{kallenberg2006foundations}
Olav Kallenberg.
\newblock {\em Foundations of modern probability}.
\newblock Springer Verlag, New York, 2002.

\bibitem{lux2016improved}
Thibaut Lux and Antonis Papapantoleon.
\newblock Improved fr{\'e}chet--hoeffding bounds on $ d $-copulas and
  applications in model-free finance.
\newblock {\em arXiv preprint arXiv:1602.08894}, 2016.

\bibitem{mai2014financial}
Jan-Frederik Mai and Matthias Scherer.
\newblock {\em Financial Engineering with Copulas Explained}.
\newblock Palgrave Macmillan, 2014.

\bibitem{nelsen1996nonparametric}
Roger~B Nelsen.
\newblock Nonparametric measures of multivariate association.
\newblock {\em Lecture Notes-Monograph Series}, pages 223--232, 1996.

\bibitem{nelsen2007introduction}
Roger~B Nelsen.
\newblock {\em An introduction to copulas}.
\newblock Springer Science \& Business Media, 2007.

\bibitem{puccetti2012computation}
Giovanni Puccetti and Ludger R{\"u}schendorf.
\newblock Computation of sharp bounds on the distribution of a function of
  dependent risks.
\newblock {\em Journal of Computational and Applied Mathematics},
  236(7):1833--1840, 2012.

\bibitem{puccetti2015computation}
Giovanni Puccetti and Ludger R{\"u}schendorf.
\newblock Computation of sharp bounds on the expected value of a supermodular
  function of risks with given marginals.
\newblock {\em Communications in Statistics-Simulation and Computation},
  44(3):705--718, 2015.

\bibitem{ruschendorf1980inequalities}
Ludger R{\"u}schendorf.
\newblock Inequalities for the expectation of $\delta$-monotone functions.
\newblock {\em Zeitschrift f{\"u}r Wahrscheinlichkeitstheorie und verwandte
  Gebiete}, 54(3):341--349, 1980.

\bibitem{ruschendorf1997optimal}
Ludger R{\"u}schendorf and Ludger Uckelmann.
\newblock {\em On optimal multivariate couplings}.
\newblock Springer, 1997.

\bibitem{ruschendorf2002n}
Ludger R{\"u}schendorf and Ludger Uckelmann.
\newblock On the n-coupling problem.
\newblock {\em Journal of multivariate analysis}, 81(2):242--258, 2002.

\bibitem{schachermayer2009characterization}
Walter Schachermayer and Josef Teichmann.
\newblock Characterization of optimal transport plans for the monge-kantorovich
  problem.
\newblock {\em Proceedings of the American Mathematical Society},
  137(2):519--529, 2009.

\bibitem{schmid2010copula}
Friedrich Schmid, Rafael Schmidt, Thomas Blumentritt, Sandra Gai{\ss}er, and
  Martin Ruppert.
\newblock Copula-based measures of multivariate association.
\newblock In {\em Copula theory and its applications}, pages 209--236.
  Springer, 2010.

\bibitem{villani2008optimal}
C{\'e}dric Villani.
\newblock {\em Optimal transport: old and new}, volume 338.
\newblock Springer Science \& Business Media, 2008.

\end{thebibliography}
\bibliographystyle{plain}

\end{document}